\documentclass[reqno]{amsart}
\usepackage{amssymb} 
\usepackage{amscd}
\usepackage{bbm}
\usepackage{color}
\usepackage{esint}
\usepackage{framed}
\usepackage{graphicx}
\usepackage{mathrsfs} 
\usepackage{hyperref} 
\usepackage{amsrefs}

\theoremstyle{plain} 
\newtheorem{theorem}{Theorem}[section]
\newtheorem{proposition}[theorem]{Proposition}

\theoremstyle{definition}

\theoremstyle{remark}

\definecolor{shadecolor}{rgb}{1,0.8,0.3}
\newcommand{\DETAIL}[1]{}
\newcommand{\IGNORE}[1]{}

\newcommand{\C}{{\mathscr{C}}}
\newcommand{\D}{{\mathscr{D}}}
\newcommand{\F}{{\mathbf{F}}}
\renewcommand{\L}{{\mathscr{L}}}
\newcommand{\M}{{\mathscr{M}}}
\newcommand{\R}{{\mathbb{R}}}
\renewcommand{\S}{{\mathcal{S}}}
\newcommand{\T}{{\mathrm{T}}}
\newcommand{\BF}{{\boldsymbol{f}}}
\newcommand{\BG}{{\boldsymbol{g}}}
\newcommand{\BH}{{\boldsymbol{h}}}
\newcommand{\BU}{{\boldsymbol{u}}}
\newcommand{\GS}{\geqslant}
\newcommand{\ID}{{\mathrm{id}}}
\newcommand{\LS}{\leqslant}
\newcommand{\MM}{{\mathsf{m}}}
\newcommand{\MU}{{\mathbf{M}}}
\newcommand{\ETA}{{\mathbf{H}}}
\newcommand{\LEB}{{\mathcal{L}}}
\newcommand{\MON}{{\mathrm{MON}}}
\newcommand{\RHO}{\varrho}
\newcommand{\SUB}[1]{
  _{\raisebox{0.25ex}{\scriptsize{$#1$}}}}
\newcommand{\SYM}{{\mathrm{sym}}}
\newcommand{\COPT}{{\mathscr{K}}}

\DeclareMathOperator*{\COF}{{\mathrm{cof}}}
\DeclareMathOperator*{\TRACE}{{\mathrm{tr}}}

\numberwithin{equation}{section}

\title 
  [Polar Cone]
  {The Polar Cone of the Set of Monotone Maps} 

\author
  {Fabio Cavalletti}
\address
  {Fabio Cavalletti,
   Lehrstuhl f\"{u}r Mathematik (Analysis),
   RWTH Aachen University,
   Templergraben 55,
   D-52062 Aachen, 
   Germany}
\email
  {cavalletti@instmath.rwth-aachen.de}

\author
  {Michael Westdickenberg}
\address
  {Michael Westdickenberg,
   Lehrstuhl f\"{u}r Mathematik (Analysis),
   RWTH Aachen University,
   Templergraben 55,
   D-52062 Aachen, 
   Germany}
\email
  {mwest@instmath.rwth-aachen.de}

\date{\today}

\begin{document}

\begin{abstract} 
We prove that every element of the polar cone to the closed convex cone of
monotone transport maps can be represented as the divergence of a measure field
taking values in the positive definite matrices.
\end{abstract}

    \maketitle \tableofcontents


\section{Introduction}
\label{S:I}

The one-dimensional pressureless gas dynamics equations
\begin{equation}
\label{E:PGD}
  \left.\begin{array}{r}
    \partial_t\RHO
      +\nabla\cdot(\RHO\BU) = 0
\\[0.5em]
    \partial_t(\RHO\BU)
      + \nabla\cdot(\RHO\BU\otimes\BU) = 0
  \end{array}\right\}
  \quad\text{in $[0,\infty)\times\R$}
\end{equation}
has recently been shown equivalent (in the regime of sticky particles) to a
first-order differential inclusion on the space of monotone transport maps from
the reference measure space $([0,1],\LEB^1|_{[0,1]|}) =: (\Omega,\MM)$ (where
$\LEB^1$ is the one-dimensional Lebesgue measure) to $\R$; see
\cite{NatileSavare}. More precisely, to every density/velocity $(\RHO,\BU)$
solving \eqref{E:PGD} one can associate a unique map $X \in \L^2(\Omega,\MM)$
with $X$ {\em monotone} such that
\begin{equation}
\label{E:COMP}
  \RHO(t,\cdot) = X(t,\cdot) \# \MM
  \quad\text{for all $t\in[0,\infty)$.}
\end{equation}
Here $\#$ indicates the push-forward of measures. Then $X$ satisfies
\begin{equation}
\label{E:FODI}
  \dot{X} + \partial I_\COPT(X) \ni \bar{V}
  \quad\text{for all $t\in[0,\infty)$,}
\end{equation}
where $\COPT$ denotes the closed convex cone of all transport maps $X \in
\L^2(\Omega,\MM)$ that are monotone, and where $\partial I_\COPT$ is the
subdifferential of the indicator function of $\COPT$. If $X$ satisfies
\eqref{E:FODI} and is related to $\RHO$ through \eqref{E:COMP}, then the
Eulerian velocity $\BU$ can be recovered from the Lagrangian velocity $V :=
\dot{X}$ through
\begin{equation}
\label{E:VELOS}
  V(t,\cdot) = \BU(t,X(t,\cdot))
  \quad\text{for all $t\in[0,\infty)$.}
\end{equation}
Assuming finite kinetic energy, it is natural to require that
$$
  V(t,\cdot) \in \L^2(\Omega,\MM),
  \quad
  \BU(t,\cdot) \in \L^2(\R,\RHO(t,\cdot)).
$$
The relation \eqref{E:VELOS} in particular determines the initial Lagrangian
velocity $\bar{V}$ in \eqref{E:FODI} in terms of the initial data $(\RHO,\BU)
(0,\cdot) =: (\bar{\RHO},\bar{\BU})$ of the system \eqref{E:PGD}.

It is shown in \cite{NatileSavare} that the solution of \eqref{E:FODI} can be
written explicitly as
\begin{equation}
\label{E:PROJ}
  X(t,\cdot) = P_\COPT( \bar{X}+t\bar{V} )
  \quad\text{for all $t\in[0,\infty)$,}
\end{equation}
with $\bar{X} := X(0,\cdot) \in\COPT$ given by \eqref{E:COMP}. Here $P_\COPT$
denotes the metric projection onto the cone $\COPT$. The connection between
\eqref{E:PGD} and \eqref{E:FODI} makes it possible to apply classical results
from the theory of first-order differential inclusions in Hilbert spaces to
study the pressureless gas dynamics equations, which form a system of hyperbolic
conservation laws. We refer the reader to \cites{NatileSavare, Brezis} for
further information.

It is known that if $X$ satisfies \eqref{E:PROJ}, then the difference $(\bar{X}
+ t\bar{V}) - X(t,\cdot)$ must be an element of the {\em polar cone} $N_\COPT(
X(t,\cdot))$ of $\COPT$, which is defined as
\begin{equation}
\label{E:POCO}
  N_\COPT(X) := \Big\{ Y \in \L^2(\Omega,\MM) \colon
    \int_\Omega Y(X'-X) \LS 0
    \quad\text{for all $X'\in\COPT$} \Big\}
\end{equation}
for all $X\in\COPT$. We observe that $N_\COPT(X)$ coincides with the
subdifferential $\partial I_\COPT(X)$. Since $\COPT$ is a cone, one can choose
$X'=2X$, $X'=0$ in \eqref{E:POCO} to obtain that
\begin{equation}
\label{E:POCO2}
  Y \in N_\COPT(X)
  \quad\Longleftrightarrow\quad
  \int_\Omega Y X = 0,
  \int_\Omega Y X' \LS 0
  \quad\text{for all $X'\in\COPT$.}
\end{equation}
One is therefore naturally led to the problem of characterizing the polar cone
of the set of monotone transport maps, beyond the basic definition
\eqref{E:POCO}. It is shown in \cite{NatileSavare} that if
$Y\in\L^2(\Omega,\MM)$ is an element of the polar cone $N_\COPT(X)$, then $Y$
coincides with the {\em derivative of a nonnegative function}. We refer the
reader to \cite{NatileSavare} for more details, and to \cites{Lions,
CarlierLachandRobert, Westdickenberg} for similar results.

In this paper, we will give a generalization of this result to the
multi-dimensional case. We are interested in the following setting: We assume
that a Borel probability measure $\RHO$ on $\R^d$ is given with finite second
moments, so that $\int_{\R^d} |x|^2 \,\RHO(dx) < \infty$. We consider the closed
convex cone of monotone transport maps
$$
  \COPT\SUB{\RHO} := \Big\{ \BF \in \L^2(\R^d,\RHO) \colon
    \text{$\BF$ is monotone} \Big\}.
$$
Here we call any Borel map $\BF \colon \R^d \longrightarrow \R^d$ monotone if
the support of the induced transport plan $\gamma_\BF := (\ID,\BF)\#\RHO$, which
is a Borel probability measure on the product space $\R^d\times\R^d$, is a
monotone set. Recall that $\Gamma \subset \R^d\times\R^d$ is monotone if
$$
  (y_1-y_2)\cdot(x_1-x_2) \GS 0
  \quad\text{for all $(x_i,y_i)\in\Gamma$ with $i=1..2$,}
$$
where $\cdot$ denotes the Euclidean inner product on $\R^d$. Our goal is to find
a representation of elements of the polar cone $\COPT\SUB{\RHO}^\perp$ (at the
zero map), defined as
$$
  \COPT\SUB{\RHO}^\perp := \Big\{ \BG\in\L^2(\R^d,\RHO) \colon
    \int_{\R^d} \BG(x)\cdot\BF'(x) \,\RHO(dx) \LS 0
    \quad\text{for all $\BF'\in\COPT\SUB{\RHO}$} \Big\}.
$$
Notice that since $\RHO$ has finite second moments, any smooth monotone function
with at most linear growth at infinity (see details below) is an element of
$\COPT\SUB{\RHO}$. Moreover, whenever $\BG\in\COPT\SUB{\RHO}^\perp$ is given,
then the product $\BG\RHO$ is in fact an $\R^d$-valued finite Borel measure,
because of Cauchy-Schwarz inequality. We will show below in
Theorem~\ref{T:STRESS} that for any $\BG \in \COPT\SUB{\RHO}^\perp$ the measure
$\BG\RHO$ can be written as the divergence of a finite Borel measure taking
values in the symmetric, positive semidefinite matrices. In the one-dimensional
case, we therefore obtain the derivative of a nonnegative function (measure) as
in \cite{NatileSavare}. Our proof relies on an application of the Hahn-Banach
theorem and is inspired by a similar argument in \cite{BouchitteGangboSeppecher}
for the construction of Michell trusses. It is possible to prove a
representation of the polar cone $\COPT\SUB{\RHO}^\perp$ similar to ours by
using a characterization of polar cones from \cite{Zarantonello} and subharmonic
functions; see \cites{Lions, Westdickenberg} for instance. Compared to these
presentations, our proof is shorter and simpler.


\section{The Main Result}
\label{S:TMR}

We will denote by $x\cdot y$ the Euclidean inner product of $x,y\in\R^k$, and by
$|x|$ the induced norm. We write $\R^{l\times l}$ for the space of real
matrices. For any $A,B \in \R^{l\times l}$ with components $A=(a_{ij})$ and
$B=(b_{ij})$ we define an inner product
$$
  \langle A,B\rangle := \TRACE(AB^\T) = \sum_{i,j=1}^l a_{i,j} b_{i,j}
$$
(with $B^\T$ the transpose of $B$), which induces the Frobenius norm
$$
  \|A\| := \sqrt{\TRACE(AA^\T)} = \sum_{i,j=1}^l a_{i,j}^2.
$$
We denote by $\S^l$ the space of symmetric real matrices and by $\S^l_+$ the
subset of positive semidefinite symmetric matrices. The space of all positive
definite, but not necessarily symmetric matrices will be denoted by $\R^{l\times
l}_+$. Recall that
$$
  A\in\R^{l\times l}_+
  \quad\Longleftrightarrow\quad
  v\cdot(Av) \GS 0
  \quad\text{for all $v\in\R^l$.}
$$
Equivalently, we have $A\in\R^{l\times l}_+$ if and only if $A^\SYM :=
(A+A^\T)/2 \in \S^l_+$.

Let $\C_*(\R^d;\R^{l\times l})$ be the space of all continuous functions $w
\colon \R^d \longrightarrow \R^{l\times l}$ with the property that
$\lim_{|x|\rightarrow\infty} w(x) \in \R^{l\times l}$ exists. Note that we can
write
$$
  \C_*(\R^d;\R^{l\times l}) = \R^{l\times l} + \C_0(\R^d;\R^{l\times l}),
$$
where $\C_0(\R^d;\R^{l\times l})$ is the closure of the space of all compactly
supported continuous $\R^{l\times l}$-valued maps, w.r.t.\ the $\sup$-norm. In
an analogous way, we define $\C_*(\R^d; \S^l)$ and $\C_*(\R^d; \S^l_+)$. For any
map $u\in\C^1(\R^d;\R^d)$ we denote by
$$
  e(u(x)) := Du(x)^\SYM 
  \quad\text{for all $x\in\R^d$}
$$
its deformation tensor, which is an element of $\C(\R^d;\S^d)$. Let
\begin{gather*}
  \C^1_*(\R^d;\R^d) := \{ u\in\C^1(\R^d;\R^d) \colon
    \text{$Du \in \C_*(\R^d;\R^{d\times d})$} \},
\\
  \MON(\R^d) := \{ u\in\C^1_*(\R^d;\R^d) \colon
    \text{$u$ is monotone} \},
\end{gather*}
so that $e(u) \in \C_*(\R^d;\S^d_+)$ if $u\in\MON(\R^d)$. The cone $\MON(\R^d)$
contains all linear maps $u(x):=Ax$ for $x\in\R^d$, where $A\in\R^{d\times
d}_+$. See \cite{AlbertiAmbrosio} for more details.

We will denote by $\M(\R^d; \R^k)$ the space of finite $\R^k$-valued Borel
measures. In an analogous way, we define $\M(\R^d;\S^l)$ and $\M(\R^d;\S^l_+)$.
If $f_i$, $i=1\ldots k$, are the components of $\F\in\M(\R^d;\R^k)$ and
$u\in\C_b(\R^d;\R^k)$ we write
$$
  \int_{\R^d} u(x)\cdot\F(dx) = \sum_{i=1}^k u_i(x)\,f_i(dx).
$$
We will say that $\F$ has finite first moment if $\sum_{i=1}^k \int_{\R^d} |x|
\,|f_i|(dx) < \infty$. If $\mu_{i,j} = \mu_{j,i}$, $i,j=1\ldots l$, are the
components of $\MU \in \M(\R^d;\S^l)$ and $v\in\C_b(\R^d;\S^l)$, then
$$
  \int_{\R^d} \langle v(x),\MU(dx) \rangle
    = \sum_{i,j=1}^l v_{i,j}(x) \,\mu_{i,j}(dx).
$$
For any $\MU=(\mu_{i,j}) \in \M(\R^d;\S^l)$ we have $\MU \in \M(\R^d;\S^l_+)$ if
and only if
$$
  \sum_{i,j=1}^l \mu_{i,j} v_i v_j
  \quad\text{is a positive measure for all $v\in\R^l$.}
$$

We can now state our representation result.

\begin{theorem}[Stress Tensor] 
\label{T:STRESS} Assume that there exist a measure $\F \in \M(\R^d;\R^d)$ with
finite first moment and a matrix-valued field $\ETA \in \M(\R^d;\S^d_+)$ with
\begin{equation}
\label{E:INEQ}
  G(u) := -\int_{\R^d} u(x) \cdot \F(dx) 
    - \int_{\R^d} \langle e(u(x)), \ETA(dx) \rangle \GS 0
\end{equation}
for all $u \in \MON(\R^d)$. Then there exists $\MU \in \M(\R^d;\S^d_+)$ such
that
\begin{align}
  G(u) &= \int_{\R^d} \langle e(u(x)), \MU(dx) \rangle
  \quad\text{for all $u \in \D(\R^d;\R^d)$,} 
\label{E:REPR}\\
  \int_{\R^d} \TRACE(\MU(dx)) 
    &\LS 
      -\int_{\R^d} x\cdot\F(dx) - \int_{\R^d} \TRACE(\ETA(dx)).
\label{E:CONTROL}
\end{align}
\end{theorem}

Notice that the integrals in \eqref{E:INEQ} are finite for any choice of $u\in
\C^1_*(\R^d;\R^d)$, by our assumptions on $\F$ and $\ETA$. Recall that the
trace of a symmetric matrix is equal to the sum of its eigenvalues, which in the
case of a positive semidefinite matrix are all nonnegative. Therefore
\eqref{E:CONTROL} controls the size of the measure $\MU$.

For $\ETA\equiv 0$ we obtain the representation announced in the introduction: 
$$
  \int_{\R^d} u(x)\cdot\F(dx) = -\int_{\R^d} \langle Du(x), \MU(dx) \rangle
$$
for all test functions $u$. Recall that $\MU$ takes values in the symmetric
matrices. The more general form of \eqref{E:INEQ} is motivated by a variational
time discretization for the compressible Euler equations, for which a
minimization problem of the form
\begin{equation}
\label{E:MIMM}
  \inf_{\BF\in\COPT\SUB{\RHO}} \bigg\{ 
    \frac{1}{2} \int_{\R^d} |\BH(x)-\BF(x)|^2 \,\RHO(dx) 
      + \int_{\R^d} e(x) \det(D\BF(x)^\SYM)^{1-\gamma} \,dx \bigg\}
\end{equation}
for suitable $\BH\in\L^2(\R^d,\RHO)$ and nonnegative $e\in\L^1(\R^d)$ must be
solved, with $\gamma > 1$ some constant. Denoting by $\BF\in\COPT\SUB{\RHO}$ the
minimizer of \eqref{E:MIMM} and letting $\BG := \BH-\BF$, we can write the
corresponding first-order optimality condition (formally) as
\begin{align*}
  & -\int_{\R^d} \BG(x)\cdot\BF'(x) \,\RHO(dx)
\\
  & \qquad
    - (\gamma-1) \int_{\R^d} e(x) \det(D\BF(x)^\SYM)^{-\gamma}
        \; \TRACE\Big( \COF(D\BF(x)^\SYM)^\T D\BF'(x) \Big) \,dx 
      \GS 0
\end{align*}
for all $\BF'\in\COPT\SUB{\RHO}$. From this, assumption \eqref{E:INEQ} follows if
we define
$$
  \F := \BG\RHO
  \quad\text{and}\quad
  \ETA := (\gamma-1) e \det(D\BF^\SYM)^{-\gamma} \COF(D\BF^\SYM)^\T.
$$
One can then check that $\F$ has finite first moments and that $\ETA \in
\M(\R^d;\S^d_+)$. This application will be discussed in more detail in an
upcoming publication.


\subsection{Positive Functionals}
\label{SS:PF}

In this section, we will discuss a general result about extensions of positive
functionals, which is due to Riedl \cite{Riedl}. Let us start with some
notation: In the following, we denote by $E$ a normed vector space. We call {\em
positive cone} any subset $C\subset E$ with $C\neq E$ with the following
properties:
\begin{equation}
\label{E:POSCONE}
  C+C \subset C,
  \qquad
  \lambda C \subset C
  \quad\text{for all $\lambda>0$,}
  \qquad
  C\cap (-C) = \{ 0 \}.
\end{equation}
The positive cone $C$ induces a partial ordering $\GS$ on the space $E$ by
$$
  y \GS x 
  \quad\Longleftrightarrow\quad
  y-x\in C.
$$
A linear map $F \colon L \longrightarrow \R$ defined on a subspace $L\subset E$
is called {\em positive} if
\begin{equation}
\label{E:POSITIVE}
  F(x) \GS 0 
  \quad\text{for all $x\in L\cap C$.}
\end{equation}
A linear map $F \colon E \longrightarrow \R$ is called {\em functional} if it is
continuous.

\begin{proposition}
\label{P:RIEDL} Let $E$ be a Banach space, partially ordered by a positive cone
$C$. If some subspace $L \subset E$ contains an interior point of $C$, then
every positive linear map $F_0 \colon L \longrightarrow \R$ can be extended to a
positive functional $F \colon E \longrightarrow \R$.
\end{proposition}

\begin{proof}
See Theorem~10.10 of \cite{Riedl}. We include a proof for the reader's
convenience. 
\medskip

{\bf Step~1.} We first observe that $E=L-C$. Indeed if $x_0\in L$ is an inner
point of $C$, then there exists a $\delta>0$ with $B_\delta(x_0) \subset C$.
Moreover, for all $x\in E$ there exists $\lambda>0$ (choose $\lambda :=
2\|x\|/\delta$, for example) with the property that
$$
  x/\lambda \subset B_\delta(0) 
    = x_0-B_\delta(x_0)
    \subset x_0-C.
$$
Since $L$ is a subspace we obtain, using $\lambda C \subset C$ for all
$\lambda>0$, that
$$
  E \subset \bigcup_{\lambda>0} \lambda(x_0-C) 
    \subset \R x_0 -C
    \subset L-C.
$$

{\bf Step~2.} Since $E=L-C$, for every $x\in E$ there exist $y_\pm\in L$ and
$z_\pm \in C$ such that $\pm x=y_\pm-z_\pm$, which implies that $y_+\GS x \GS
-y_-$. We now define
\begin{equation}
\label{E:PX}
  p(x) := \inf \Big\{ F_0(y) \colon y\in L, y\GS x \Big\}
  \quad\text{for all $x\in E$.}
\end{equation}
Then $p(x) \LS F_0(y_+) < \infty$. On the other hand, for every $y\in L$ with
$y\GS x$ we have $y\GS-y_-$. Since $y+y_- \in L\cap C$, we have $F_0(y+y_-)\GS
0$, by positivity of $F_0$. This implies that $F_0(y) \GS -F_0(y_-) > -\infty$.
We conclude that $p(x)$ is finite for all $x\in E$. It is easy to check that for
all $x_1,x_2\in E$ and for all $\lambda>0$ we have
$$
  p(x_1+x_2) \LS p(x_1) + p(x_2), 
  \qquad 
  p(\lambda x_1) = \lambda p(x_1).
$$
For every $x\in L$ and $z\in E$ with $z\GS x$, we have $F_0(x) \LS p(z)$ (in
particular, we may choose $z=x$). Indeed for every $y\in L$ with $y\GS z$, we
have $y\GS x$, thus $y-x \in L\cap C$. Hence $F_0(y-x)\GS 0$, by positivity,
which yields $F_0(y) = F_0(x) + F_0(y-x) \GS F_0(x)$. Taking the $\inf$ over all
$y\in L$ with $y\GS z$, we obtain the estimate. 
\medskip

{\bf Step~3.} We can now apply the Hahn-Banach theorem and obtain a linear map $F
\colon E\longrightarrow \R$ with $F(x) \LS p(x)$ for all $x\in E$. In order to
show that $F$ is positive, let $x\in C$. Then $0\GS -x$ and $0\in L$, so we may
choose $y=0$ in the definition of $p(-x)$ (see \eqref{E:PX}) to obtain $p(-x)
\LS 0$. Therefore $F(-x) \LS p(-x) \LS 0$, and so $F(x) \GS 0$ for all $x\in C$.
To prove that $F$ is an extension of $F_0$, let $x\in L$. Then we may choose
$y=-x$ in \eqref{E:PX} to obtain $p(-x) \LS F_0(-x)$ for all $x\in L$. Then
$$
  -F(x) = F(-x) \LS p(-x) \LS F_0(-x) = -F_0(x),
$$
hence $F_0(x) \LS F(x)$. Applying the same argument to $-x \in L$, we get
$F_0(x) \GS F(x)$. It follows that $F_0(x) = F(x)$ for all $x\in L$. Therefore
$F$ is an extension of $F_0$. 
\medskip

{\bf Step~4.} It remains to prove that $F$ is continuous. Let $x_0$ be the
interior point of $C$ from Step~1, for which $B_\delta(x_0) \subset C$. Then
$B_\delta(0) \subset \pm(x_0-C)$. Let $\lambda := F(x_0) \GS 0$ (recall that
$F(x)\GS 0$ for all $x\in C$). Then for all $x\in B_\delta(0)$ we have $x_0-x\in
C$, thus $F(x_0-x)\GS 0$. It follows that $F(x_0) \GS F(x)$. Similarly, we
obtain $F(x) \GS -F(x_0)$. Then either $F$ vanishes (if $\lambda=0$), or the
preimage of the nonempty interval $(-\lambda,\lambda)$ contains a neighborhood
of $0$, and so $F$ (being linear) is continuous.
\end{proof}


\subsection{Proof of Theorem~\ref{T:STRESS}}
\label{SS:POTR}

We apply Proposition~\ref{P:RIEDL} with
$$
  E := \C_*(\R^d;\S^d),
  \quad
  C := \C_*(\R^d,\S^d_{+}),
  \quad  
  L := \{ e(u) \colon u\in\C^1_*(\R^d;\R^d) \}.
$$
Clearly $C$ satisfies conditions \eqref{E:POSCONE}. The identity map $\ID$ is an
element of $\MON(\R^d)$, with constant deformation tensor $e(\ID)$ equal to the
identity matrix $\boldsymbol{1} \in \S^d_+$. Since the eigenvalues of a
symmetric matrix depend continuously on the matrix entries, we have that $e(\ID)
= \boldsymbol{1}$ is an interior point of $C$: For all $\|v-\ID\|_E$
sufficiently small, the eigenvalues of $v(x)$ are bigger than $1/2$ for all
$x\in\R^d$ and $v\in E$.

On the subspace $L \subset E$, we define the functional $F_0$ as
$$
  F_0(v) := -\int_{\R^d} u(x) \cdot \F(dx) 
    - \int_{\R^d} \langle v(x), \ETA(dx) \rangle
  \quad\text{where $v=e(u)$.}
$$
Note that $F_0$ is well-defined: If there exists another map
$\tilde{u}\in\C^1_*(\R^d;\R^d)$ such that $e(\tilde{u}(x))=v(x)$ for all
$x\in\R^d$, then we have $e(u-\tilde{u})\equiv 0$, by linearity. Consequently,
there exist an antisymmetric matrix $B\in\R^{d\times d}$ and $c\in\R^d$ such
that
$$
  \bar{u}(x) := u(x)-\tilde{u}(x) = Bx+c
  \quad\text{for all $x\in\R^d$.}
$$
Indeed assume that $e(\bar{u}(x)) = 0$ and define
$$
  W\bar{u}(x) : = \frac{D\bar{u}(x)-D\bar{u}(x)^\T}{2}
  \quad\text{for all $x\in\R^d$.}
$$
Then $\partial_{k} (W\bar{u})_{i,j} \equiv 0$ for all indices $i,j,k$. Since
$D\bar{u} = e(\bar{u}) + W\bar{u}$ it follows that $D\bar{u}$ is a constant
matrix with vanishing symmetric part, so $\bar{u}$ is a rigid deformation. We
now observe that both $\pm\bar{u}\in\MON(\R^d)$, which implies $F_0(e(\bar{u}))
= 0$ because of \eqref{E:INEQ}. As $F_0$ is linear, we conclude that $F_0$ is
well-defined. Similarly, one can check that $F_0(v) \GS 0$ for all $v\in L\cap
C$, so the linear map $F_0 \colon L \longrightarrow \R$ is positive.

Applying Proposition~\ref{P:RIEDL}, we obtain that $F_0$ can be extended to a
continuous linear map $F \colon E \longrightarrow \R$. Notice that $\C_*(\R^d;
\R)$ is a separable and {\em closed subalgebra} of the space $\C_b(\R^d;\R)$
of bounded continuous $\S^d$-functions. As is well-known, to any closed
subalgebra of a space of bounded continuous functions, there corresponds a
compactification of the domain. In our case, we obtain the one-point (also
called Alexandroff) compactification of $\R^d$, which we will denote by
$\beta\R^d$. Then $\C_*(\R^d; \S^d)$ is isomorphic to $\C(\beta\R^d; \S^d)$. We
refer the reader to \cite{Folland} Section~4.8 for more details. By the Riesz
representation theorem, there therefore exists a finite Radon measure $\MU \in
\M(\beta\R^d; \S^d)$ that represents the functional $F$ in the sense that
$$
  F(v) = \int_{\beta\R^d} \langle v(x),\MU(dx) \rangle
  \quad\text{for all $v\in\C_*(\R^d;\S^d)$.}
$$
Since $F(v)\GS 0$ for all $v\in\C_*(\R^d;\S^d_+)$ we obtain that $\MU$ takes in
fact values in $\S^d_+$. Moreover, as $F$ is an extension of $F_0$, the
following identity holds:
$$
  F_0(v) = -\int_{\R^d} u(x) \cdot \F(dx) 
      - \int_{\R^d} \langle v(x), \ETA(dx) \rangle
    = \int_{\beta\R^d} \langle v(x), \MU(dx) \rangle
$$
for any $v=e(u)$ and $u\in\C^1_*(\R^d;\R^d)$; see \eqref{E:REPR}. In particular,
we may choose $u=\ID$ (with $e(\ID)=\boldsymbol{1})$ to obtain the control
(recall that $\MU$ is $\S^d_+$-valued)
$$
  \int_{\beta\R^d} \TRACE(\MU(dx)) 
    = -\int_{\R^d} x\cdot\F(dx) - \int_{\R^d} \TRACE(\ETA(dx)).
$$
Restricting the representation from $\beta\R^d$ to $\R^d$, we obtain the result.


\end{document}